\documentclass[11pt]{amsart}

\usepackage[T1]{fontenc}
\usepackage{amsfonts,mathrsfs}
\usepackage{amssymb, amsmath}
\usepackage{times}

 \theoremstyle{plain}
\newtheorem{claim}{\sc Claim}[section]

\newtheorem{lemma}[claim]{\sc lemma}
\newtheorem{proposition}[claim]{\sc Proposition}
\newtheorem{theorem}[claim]{\sc Theorem}

\newtheorem{conjecture}[claim]{\sc Conjecture}
\newtheorem{definition}[claim]{\sc Definition}

\newtheorem{example}[claim]{\sc Example
}

\begin{document}
\title{Polite actions of non-compact Lie groups}
\author{Larry Bates and J\k{e}drzej \'{S}niatycki}

\begin{abstract}
\noindent Based mainly on examples of interest in mechanics, we define the
notion of a polite group action. One may view this as not only trying to
give a more general notion than properness of a group action, but also to
more fully understand the role of invariant functions in describing just
about everything of interest in reduction.

We show that a polite action of a symmetry group of a dynamical system admits
reduction and reconstruction.
\end{abstract}

\maketitle

Dirac's seminal 1950 paper (\cite{dirac50}) showed how to construct a
reduced bracket on a Hamiltonian system with constraints, but did not
focus on constraints generated by the action of a symmetry group. The
first significant theory of reduction of a Hamiltonian system with
symmetry was given by Meyer in 1973 (\cite{meyer}), and this was
followed by work of Marsden and Weinstein a year later
(\cite{marsden-weinstein}.) Since then there has been a veritable
flood of papers endeavouring to understand reduction and various forms
of singular behaviour. For example, many of these works have studied
what happens when the action of the symmetry group is not free and
quotient spaces are not manifolds. It is probably fair to say that a
reasonably complete reduction theory now exists in the case that the
group action is proper (see, for example, \cite{bates-lerman},
\cite{cushman-bates}, \cite{ortega-ratiu},\cite{sjamaar-lerman},
\cite{sniatycki13}.) Here we make the case that since there are
interesting, important examples in mechanics where the symmetry group
does not act properly, a less restrictive notion of group action
warrants consideration.

This paper defines the notion of a \textit{polite} action, and gives
some examples. In addition, it proves that a polite action of the
symmetry group of a dynamical system admits reduction and
reconstruction.  This means that the dynamical vector field projects
to a vector field on a reduced space, and that the original dynamics
can be recovered from the dynamics on the reduced space.  This is all
done in the context of vector fields and differential equations on
manifolds.

Since the possibility exists that our notion of a polite action is not
the last word on group actions in mechanics, we hope that, in the
spirit of this commemorative volume, others will provide even better
solutions to the problem of `what's next'.

\section{Motivating examples}

\noindent The following examples motivate why one needs to deal with
problems where the group action is not proper, so that strictly speaking the
usual reduction theories do not apply.

\begin{enumerate}
\item The one-dimensional harmonic oscillator. Here the Hamiltonian is $h(p,q) = \frac12 p^2 +\frac12 q^2$ on the phase space $P=T^*\mathbb{R}$. All
solutions of Hamilton's equations are periodic with period $2\pi$. The
Hamiltonian flow $\phi_t$ is 
\begin{equation*}
\phi_t 
\begin{pmatrix}
q_0 \\ 
p_0
\end{pmatrix}
= 
\begin{pmatrix}
\cos t & \sin t \\ 
-\sin t & \cos t
\end{pmatrix}
\begin{pmatrix}
q_0 \\ 
p_0
\end{pmatrix}
.
\end{equation*}
The action of $\mathbb{R}$ on $P$ is not proper but is indistinguishable
from the free proper action of the compact group $\mathbb{R}/2\pi\mathbb{Z}$.

\item The stiff spring. The Hamiltonian is, for $\epsilon >0$, 
\begin{equation*}
h(q,p) = \frac12 p^2 +\frac12 q^2 + \frac{\epsilon}{4} q^4.
\end{equation*}
Hamilton's equations yield Duffing's equation $q^{\prime\prime} + q +
\epsilon q^3=0$. This implies that the solution may be written in terms of
the Jacobi elliptic function $\mathrm{cn}$ as
\begin{equation*}
q(t) = \mathop{\mathrm{cn}}\nolimits\left( \sqrt{1+\epsilon}\, t; \sqrt{
\frac{\epsilon}{2(1+\epsilon)}}\right).
\end{equation*}
Here the parameters are chosen so that $q(t)$ solves the initial value problem 
\begin{equation*}
q^{\prime\prime} + q +\epsilon q^3 = 0, \qquad q(0)=1,\quad q^{\prime}(0)=0,
\end{equation*}
for $\epsilon >0$. It follows that the period $\tau$ is 
\begin{align*}
\tau & = \frac{4}{\sqrt{1+\epsilon}} \, K\left(\sqrt{\frac{\epsilon}{2(1+\epsilon)}}\right), \\
& = 2\pi\left(1 - \frac{3}{8}\epsilon + \frac{57}{256}\epsilon^2 + \cdots
\right).
\end{align*}
where $K(k)$ is the complete elliptic integral of the first kind. It is now
easy to solve for other initial conditions to find the period as a function
of the energy $h$ and $\epsilon$. The Hamiltonian flow $\phi_t$, which is an
action of $\mathbb{R}$ on $P=T^*{\mathbb{R}}$ is still periodic, but not
proper. In this case there is no \textit{fixed} subgroup $G$ of $\mathbb{R}$
with the flow $\phi_t$ being a proper $\mathbb{R}/G$ action (although we can
do this individually for each orbit.) However, it is common practice in
mechanics to rescale the Hamiltonian vector field $X_h$ by the period $\tau$
to produce a new vector field $Y=\tau X_h$, all of whose integral curves are
periodic of period 1. It is a theorem that the resulting vector field $Y $
is still a Hamiltonian vector field, and we produce a new variable called
the action (see, for example \cite{bates-sniatycki92a}.) In this way a free
proper action $\psi_t$ of the compact group $\mathop{\! \, \rm SO }%
\nolimits(2)$ is associated to the original nonproper action $\phi_t$ by
setting $\psi_t:= \phi_{\tau t}$.

\item The champagne bottle. The Hamiltonian in this case is 
\begin{equation*}
h= \frac12(p_1^2+p_2^2) +(q_1^2+q_2^2)^2 - (q_1^2+q_2^2)
\end{equation*}
on the phase space $P=T^*\mathbb{R}^2$. This is a completely integrable
system because of the rotational invariance. The Hamiltonian $h$, together
with the angular momentum $j$ gives the construction of action variables $
(I_1,I_2)$ that generate a torus action whose orbits contain the original
quasiperiodic trajectories of the Hamiltonian. In this way, a proper group
action is associated to the non-proper Hamiltonian action of ${\mathbb{R}}^2$
associated to the flow of the commuting Hamiltonian vector fields of the
energy and the angular momentum (see \cite{bates-sniatycki92} for more
details.) However, what is interesting in this case is that the construction
of the actions is only local because of the presence of an obstruction
called monodromy preventing the torus group action being globally
well-defined (see \cite{bates91}.)

\item A nonabelian example. We construct an oriented $S^3$ bundle over ${\mathbb{R}}^3\backslash \{0\}$. The fiber $S^3$ is diffeomorphic to the
group $\mathop{\! \, \rm Spin }\nolimits(3)$, but the bundle is not a
principal bundle. In a sense, we may view this example as a simply-connected
version of the previous example.

To start, consider the two copies of the trivial bundle $D^2 \times S^3$,
which we think of as local trivializations of our bundle over the upper and
lower hemispheres of the sphere $S^2$. Viewing $S^3$ as the unit sphere in $
\mathbb{R}^4$, we consider the gluing map from one hemisphere to another as
a map from the equator into Diff$^+$($S^3$). By a theorem of Hatcher \cite{hatcher}, this diffeomorphism group retracts onto the orthogonal group $\mathop{\! \, \rm SO }\nolimits(4)$. The orthogonal group is diffeomorphic
to the product $\mathop{\! \, \rm SO }\nolimits(3)\times%
\mathop{\! \, \rm
Spin }\nolimits(3)$, and has fundamental group $\mathbb{Z}_2$. The
transition map from one hemisphere to the other is given by the map 
\begin{equation*}
S^1 \longrightarrow \mathop{\! \, \rm SO }\nolimits(4):\phi \longrightarrow 
\begin{pmatrix}
\cos \phi & -\sin \phi & 0 & 0 \\ 
\sin \phi & \cos \phi & 0 & 0 \\ 
0 & 0 & 1 & 0 \\ 
0 & 0 & 0 & 1
\end{pmatrix}
.
\end{equation*}
This map is a generator of the fundamental group of $\mathop{\! \, \rm SO }
\nolimits(4)$ because the matrix represented by the upper left $2\times2$
block is a generator of the fundamental group of $\mathop{\! \, \rm SO }
\nolimits(2)$, and we have the natural inclusions 
\begin{equation*}
\mathop{\! \, \rm SO }\nolimits(2) \hookrightarrow \mathop{\! \, \rm SO }
\nolimits(3) \hookrightarrow \mathop{\! \, \rm SO }\nolimits(4)
\end{equation*}
and thus a surjection in homotopy $\pi_1(\mathop{\! \, \rm SO }\nolimits(2))
\longrightarrow \pi_1(\mathop{\! \, \rm SO }\nolimits(4))$. This implies
that the bundle is not a trivial bundle.

Observe that the south pole $(0,0,0,1)$ on the sphere $S^3$ is fixed by the
transition map, and this implies that the map $S^2 \longrightarrow$ `south
pole' is a global section of the bundle. This fact, together with the
nontriviality of the bundle implies that the bundle is not a principal $
\mathop{\! \, \rm Spin }\nolimits(3)$ bundle, as any principal bundle with a
global section must be globally trivial.

Reviewing this example from the point of view of classifying spaces suggests
that many more such examples may be constructed by considering $\mathop{\!
\, \rm Spin }\nolimits(3)$ bundles over the four-sphere $S^4$.

The bundle constructed here may be given a symplectic structure by embedding
the sphere $S^2$ into $\mathbb{R}^3\backslash 0$ in the usual way. In more
detail, let $S^2$ be $x_1^2+x_2^2+x_3^2=1$, and $\psi_1,\psi_2,\psi_3$ be
the usual left-invariant one-forms on $\mathop{\!
\, \rm Spin }\nolimits(3)$. Then the form 
\begin{equation*}
\omega = \psi_1\wedge \psi_2 + d(z(x_3)\psi_3) + dx_1\wedge dx_2
\end{equation*}
is a symplectic form on our bundle where $z(x_3)$ is a function that
satisfies 1) $z^{\prime}(x_3) >0$ for all $x_3$, and 2) $|z(x_3)| <1$ for
all $x_3$. For example, we may take $z(x)= x/\sqrt{1+x^2}$.

\item Consider the Hamiltonian system given by the motion of the free
particle in space (you can take any dimension $n\geq2$ for space.) The
Euclidean group $\mathop{\! \, \rm SE }\nolimits(n)$ acts in a Hamiltonian
way on the phase space $T^*{\mathbb{R}}^n$ and preserves the level set $h^{-1}(1/2)$, which are the straight lines parametrized by arclength. We are
of course taking the Hamiltonian to be $h=|p|^2/2$. The components of the
momentum map for the Euclidean group are the linear and angular momentum,
and as they commute with the Hamiltonian, they pass to an action on the
quotient space $\bar{P}:=h^{-1}(1/2)/\sim $, where the $\sim$ represents the
quotient by the Hamiltonian flow $\phi_t(q,p)=(q+tp,p)$. The quotient
manifold $\bar{P}$, which is the space of oriented lines in ${\mathbb{R}}^n$
, is naturally endowed with a symplectic structure, as follows from the
reduction theorem. Furthermore, the action of the Euclidean group on the
quotient $\bar{P}$ is Hamiltonian. This action is not fixed point free and
is not proper, as the stability subgroup of a point in the quotient contains
the subgroup which corresponds to the translations along the line that it
represents. More precisely, the space of lines is the homogeneous space $\mathop{\! \, \rm SE }\nolimits(n)/(\mathop{\! \, \rm SO }\nolimits(n)
\times {\mathbb{R}}) \sim T^*S^n$. This construction is used when studying
the Radon transform, as it involves integration along lines.
\end{enumerate}

\section{Polite actions}

Consider a Hamiltonian system $(P,\omega ,h)$ invariant under the action $\phi$ of a connected Lie group $G$. Given a closed subroup $H$ of $G$,
define 
\begin{equation*}
P_{H}:=\{p\in P\mid G_{p}=H\}.
\end{equation*}
Denote by $N^{H}$ the normalizer of $H$ in $G$; that is 
\begin{equation*}
N^{H}=\{n\in G\mid n^{-1}hn\in H\text{ for all }h\in H\}.
\end{equation*}
The normalizer is a closed subgroup of $G$.

\begin{lemma}
The action of $N^{H}$ on $P$ preserves $P_{H}$.
\end{lemma}

\begin{proof}
For $p\in P_{H}$, and $n\in N^{H}$, the isotropy group $G_{np}$ of $np$ is
given by 
\begin{eqnarray*}
G_{np} &=&\{g\in G\mid gnp=np\} \\
&=&\{g\in G\mid n^{-1}gnp=p\} \\
&=&\{g\in G\mid n^{-1}gn\in H\}.
\end{eqnarray*}
In other words, $g\in G_{np}$ if and only if $h=n^{-1}gn\in H$. Therefore, $g=nhn^{-1}\in H$, and $G_{np}=H.$ Hence, $np\in P_{H}$. Thus, the action of $N^{H}$ on $P$ preserves $P_{H}$.
\end{proof}

Since the action of $N^{H}$ on $P$ preserves $P_{H}$, it induces an action
of $N^{H}$ on $P_{H}.$ Let $G_{H}=N^{H}/H$. Since $H$ is closed in $G$, it
is closed in $N^{H}$, and $G_{H}$ is a Lie group. Moreover, there is an
action 
\begin{equation*}
G_{H}\times P_{H}\rightarrow P_{H}:([n],p)=np,
\end{equation*}
where $[n]$ is the equivalence class of $n$ in $G_{H}=N^{H}/H$.

\vspace{11pt}

Warning: The group $G_H$ is \textit{not} a subgroup of the group $G$.

\begin{proposition}
\label{0}The action of $G_{H}$ on $P_{H}$ is free.
\end{proposition}

\begin{proof}
For $g\in N^{H}$, suppose $[g]\in G_{H}$ preserves a point $p\in P_{H}$;
that is $gp=p$. This means that $g\in G_{p}=H$. Therefore, $[g]$ is the
identity in $G_{H}$.
\end{proof}

\begin{definition}
The action of $G$ on $P$ is polite if for each closed subgroup $H$ of $G$,
the set $P_{H}$ is a manifold and the action of $G_{H}$ on $P_{H}$ is proper.
\end{definition}

\section{Examples of polite actions}

It is straightforward to check that the group action in all of the following
examples is polite.

\begin{example}
The actions in the motivating examples 1,2,3,5 are all polite.
\end{example}

\begin{example}
Every action of a compact group is polite because it is proper.
\end{example}

\begin{example}
The ${\mathbb{R}}$ action generated by the flow of the vector field $X= \sin
x\, \partial_x + \cos x \, \partial_y $ on the plane is free but not polite.
\end{example}

\begin{example}
The coadjoint action of a compact connected Lie group is polite because this
is just the action of the (finite) Weyl group $N^H/H$ on a coadjoint orbit.
\end{example}

\begin{example}
The co-adjoint action of $\mathop{\! \, \rm SL }\nolimits(2,{\mathbb{R}})$.
\end{example}

There are three co-adjoint orbits of interest, which we can label as
parabolic, hyperbolic and elliptic since the group is semi-simple. The
elliptic and hyperbolic ones correspond to Cartan subalgebras, so the
corresponding stability groups are self-normalizing. The only nontrivial
case is the parabolic one, and here the normalizer is the Borel subgroup
which we may take as the upper triangular matrices. The quotient $N^H/H$ in
this case acts by dilations on the cone (translations along the ruling), and
the action is again seen to be proper.

\begin{example}
A special class of solvable groups of type $\mathfrak{S}$.
\end{example}

Following \textsc{nomizu} \cite{nomizu}, we say that a group $G$ belongs to
the class $\mathfrak{S}$ if the Lie algebra $\mathfrak{g}$ of the Lie group $G$ contains a codimension one commutative ideal $\mathfrak{a}$ and an
element $Y$ with the property that $[Y,X]=X$ for all $X$ in the ideal $\mathfrak{a}$. Let $X_{1},\dots ,X_{n}$ be a basis for $\mathfrak{a}$, and
let $X_{1}^{\ast },\dots ,X_{n}^{\ast },Y^{\ast }$ be the corresponding dual
basis in the dual space $\mathfrak{g}^{\ast }$. Let $(a,b)$ be an element in
the half-space ${\mathbb{R}}^{+}\times {\mathbb{R}}$, and consider the point 
$\mu =aX_{1}^{\ast }+bY^{\ast }\in \mathfrak{g}^{\ast }$. Then the non-zero
infinitesimal generators of the co-adjoint action are generated by 
\begin{equation*}
\mathop{\mathrm{ad}}\nolimits_{X_{1}}^{\ast }|_{\mu }=-a\frac{\partial }{\partial Y^{\ast }},\qquad \mathop{\mathrm{ad}}\nolimits_{Y}^{\ast }|_{\mu
}=-a\frac{\partial }{\partial X_{1}^{\ast }}.
\end{equation*}%
This implies that the co-adjoint orbit through the point $\mu $ is the
two-dimensional open half plane spanned by $Y^{\ast }$ and $\mu $. It
follows that the Lie algebra $\mathfrak{h}=\mathop{\mathrm{span}}
\nolimits\{X_{2},\dots ,X_{n}\}$, and hence that the isotropy group $H\sim {\mathbb{R}}^{n-1}$. Thus the normalizer $N^{H}=G$ and 
\begin{equation*}
N^{H}/H\sim \mathop{\! \, \rm Aff \!}\nolimits^{\,+}(1,{\mathbb{R}})
\end{equation*}
acts freely, transitively and properly on the co-adjoint orbit through $\mu $. Note that the action is just the usual action of the affine group on the
half-plane.

In light of the previous two examples we make the following

\begin{conjecture}
The coadjoint action of a Lie group on the dual of its Lie algebra is polite
for any group in which the coadjoint orbits are locally closed.
\end{conjecture}

\section{Reduction and reconstruction of polite symmetries}

We consider a dynamical system given by a smooth vector field $X$ on a
manifold $P$, called the phase space of the system. Evolutions of our
dynamical system are integral curves $\gamma :I\rightarrow P$ of $X$, where $I$ is an interval in $\mathbb{R}$. 

Let $\Phi :G\times P\rightarrow P$ be an action of a Lie group $G$ on $P$.
We say that $G$ is a symmetry group of our dynamical system if the action $\Phi $ preserves the vector field $X$. The \textit{reduced phase space} is
the space $\bar{P}=P/G$ of $G$-orbits in $P$ endowed with a differential
structure 
\begin{equation*}
C^{\infty }(\bar{P})=\{f:\bar{P}\rightarrow \mathbb{R}\mid \rho ^{\ast }f\in
C^{\infty }(P)^{G}\},
\end{equation*}%
where $\rho :P\rightarrow \bar{P}$ is the orbit map and $C^{\infty }(P)^{G}$
is the ring of $G$-invariant smooth functions on $P.$ It should be noted
that the orbit space $\bar{P}$ has two topologies: the quotient space
topology and the differential space topology. Here, we take the differential
space topology.\footnote{For applications of the theory of differential spaces to reduction of
symmetries see \cite{sniatycki13}. } The reduced dynamical system is the
derivation $\rho _{\ast }X$ of $C^{\infty }(\bar{P})$ defined by 
\begin{equation}
\rho ^{\ast }((\rho _{\ast }X)(f))=X(\rho ^{\ast }f)  \label{push-forward}
\end{equation}%
for every $f\in C^{\infty }(\bar{P}).$  

\begin{proposition}
For every integral curve $\gamma :I\rightarrow P$ of $X$, the curve $\rho
\circ \gamma :I\rightarrow \bar{P}:t\mapsto \rho (\gamma (t))$ satisfies the
equation
\begin{equation}
\frac{d}{dt}f(\rho (\gamma (t)))=((\rho _{\ast }X)(f))(\rho (\gamma (t)))
\label{reduced equation}
\end{equation}
for and each $f\in C^{\infty }(\bar{P})$ and $t\in I$. 
\end{proposition}

\begin{proof}
It follows from equation (\ref{push-forward}) that
\begin{eqnarray*}
\frac{d}{dt}f(\rho (\gamma (t))) &=&\frac{d}{dt}((\rho ^{\ast }f)(\gamma
(t))=(X(\rho ^{\ast }f))(\gamma (t)) \\
&=&\rho ^{\ast }((\rho _{\ast }X)(f))(\gamma (t))=((\rho _{\ast }X)(f))(\rho
(\gamma (t))).
\end{eqnarray*}
\end{proof}

Equation (\ref{reduced equation}) is called the \textit{reduced equation.} A
curve $\rho \circ \gamma :I\rightarrow \bar{P}$ satisfying the reduced
equation gives a reduced evolution of the system. Given a reduced evolution $\bar{\gamma}:I\rightarrow \bar{P}$ of the system, the process of finding
integral curves $\gamma $ of $X$ such that $\rho \circ \gamma =\bar{\gamma}$
is called \textit{reconstruction}. If the action $\Phi $ of $G$ on $P$ is
free and proper, the reduced equation as well as equations involved in
reconstruction are ordinary differential equations on manifolds.

\begin{definition}
An action $\Phi :G\times P\rightarrow P$ that preserves a vector field $X$
on $P$ admits reduction and reconstruction if the reduced equation and
equations involved in reconstruction can be presented as differential
equations on manifolds.
\end{definition}

\begin{theorem}
\label{2}A polite action $\Phi :G\times P\rightarrow P$ that preserves a
vector field $X$ on $P$ admits reduction and reconstruction.
\end{theorem}

We shall prove this theorem by a sequence of propositions.

\begin{proposition}
\label{1 copy(1)}Let $X$ be a vector field on $P$ that is invariant under a
polite action of a Lie group $G$ on $P$. For each closed subgroup $H$ of $G$, the flow of $X$ preserves $P_{H}$.
\end{proposition}

\begin{proof}
Let $\exp tX$ be the local one-parameter group of local diffeomorphisms of $P
$ generated by $X$, and $H$ be a closed subgroup of $G$. For each $g\in H$
we have 
\begin{equation*}
g\exp tXg^{-1}=\exp tX,
\end{equation*}
because $X$ is $G$-invariant and $H\subseteq G$. Hence, for each $p\in P_{H}$
and $g\in H$, 
\begin{equation*}
g\exp tX(p)=\exp tXgp=\exp tX(p),
\end{equation*}
which implies that $\exp tX(p)\in P_{H}.$
\end{proof}

Let $\bar{P}_{H}=\rho (P_{H})$ and $\rho _{H}:P_{H}\rightarrow \bar{P}_{H}$
be the restriction of $\rho $ to $P_{H}$. The following diagram 
\begin{equation*}
\begin{array}{ccccc}
&  & \iota _{H} &  &  \\ 
& P_{H} & \hookrightarrow  & P &  \\ 
\rho _{H} & \downarrow  &  & \downarrow  & \rho  \\ 
& \bar{P}_{H} & \hookrightarrow  & \bar{P} &  \\ 
&  & \epsilon _{H} &  & 
\end{array}
\end{equation*}
where the horizontal arrows are the inclusion maps, commutes.

The space $\bar{P}_H$ has the differential structure 
\begin{equation*}
C_1^{\infty }(\bar{P}_H)=\{h:\bar{P}_H\rightarrow \mathbb{R\mid \rho }
_H^{\ast}h\in C^{\infty }(P_H)\}
\end{equation*}
and a differential structure $C_2^{\infty }(\bar{P}_H)$ generated by the
restrictions to $\bar{P}_H$ of smooth functions on $\bar{P}$.

\begin{proposition}
The differential structures $C_{2}^{\infty }(\bar{P}_{H})$ and $C_{1}^{\infty }(\bar{P}_{H})$ are related by the inclusion 
\begin{equation*}
C_{2}^{\infty }(\bar{P}_{H})\subseteq C_{1}^{\infty }(\bar{P}_{H}).
\end{equation*}
If the action of $G$ on $P$ is improper, $C_{2}^{\infty }(\bar{P}_{H})$ may
be a proper subset of $C_{1}^{\infty }(\bar{P}_{H}).$
\end{proposition}

\begin{proof}
If $f\in C^{\infty }(\bar{P})$, then $\epsilon _H^{\ast }f=f_{\mid \bar{P}
_{H}}\in C_{2}(\bar{P}_{H})$. On the other hand, $\rho ^{\ast }f\in
C^{\infty }(P)$ and the restriction of $\rho ^{\ast }f$ to $P_{H}$ is an $N^{H}$-invariant smooth function $(\rho ^{\ast }f)_{\mid P_{H}}=\iota
_{H}^{\ast }\rho ^{\ast }f$ on $P_{H}$. Moreover, $\rho \circ \iota
_{H}=\epsilon _{H}\circ \rho _{H}$ implies that $\iota _{H}^{\ast }\rho
^{\ast }f=\rho _{H}^{\ast }\epsilon _{H}^{\ast }f$. Therefore, $\epsilon
_{H}^{\ast }f\in C_{1}^{\infty }(\bar{P}_{H}).$

Suppose now that $h:\bar{P}_{H}\rightarrow \mathbb{R}$ is such that, for
every $r\in \bar{P}_{H}$, there exists a neighbourhood $U_{r}$ of $r$ in $\bar{P}_{H}$ and a function $f_{r}\in C^{\infty }(\bar{P})$ such that $\epsilon _{H}^{\ast }f_{r\mid U_{r}}=h_{\mid U_{r}}$. By definition of the
differential structure generated by a family of functions, $f_{r}\in
C_{2}^{\infty }(\bar{P}_{H})$. We have shown above that $\epsilon _{H}^{\ast
}f_{r}\in C_{1}^{\infty }(\bar{P}_{H})$. Hence, $f_{r\mid \bar{P}_{H}\cap
U_{r}}=f_{r\mid U_{r}}$, which implies 
\begin{equation*}
C_{2}^{\infty }(\bar{P}_{H})\subseteq C_{1}^{\infty }(\bar{P}_{H}).
\end{equation*}

On the other hand, suppose that $h\in C_{1}^{\infty }(\bar{P}_{H}),$ which
means that $\mathbb{\rho }_{H}^{\ast }h\in C^{\infty }(P_{H})^{H}$. The set 
\begin{equation*}
P_{(H)}=\{gp\in P\mid g\in G,\text{\ }p\in P_{H}\}
\end{equation*}
is the union of the orbits of $G$ through points in $P_{H}$. We can extend
the $H$-invariant function $\mathbb{\rho }_{H}^{\ast }h$ on $P_{H}$ to a $G$-invariant function $k$ on $P_{(H)}.$ If the action of $G$ on $P$ is not
proper, we have no guarantee that a $G$-invariant function $k$ on $P_{(H)}$
extends to a $G$-invariant function on $P$, as may be seen in the following
example. Let $X$ be the planar vector field 
\begin{equation*}
X= \sin x\, \partial_x + \cos x \, \partial_y.
\end{equation*}
Since $X$ has bounded norm in the plane (so has a complete flow) and is
invariant by translations of $2\pi$ in both the $x$ and $y$ directions, $X$
generates an ${\mathbb{R}}$-action on the torus ${\mathbb{R}}^2/(2\pi{\mathbb{Z}}\times 2\pi{\mathbb{Z}})$ This action is not free only on the two
circular orbits through $[(0,0)]$ and $[(\pi,0)]$, and the isotropy group $H 
$ for these orbits in this case is $2\pi{\mathbb{Z}}$. Any function that is
locally constant on each circle need not extend to an invariant function on
the entire torus unless it has the same value on each circle, because the
pair of circles are the alpha and omega limit sets of every other trajectory
on the torus.

Hence, if the action is not proper, $C_{2}^{\infty }(\bar{P}_{H})$ may be a
proper subset of $C_{1}^{\infty }(\bar{P}_{H}).$
\end{proof}

In the following we shall consider $\bar{P}_{H}$ with the differential
structure $C_1^{\infty }(\bar{P}_H)$.

\begin{proposition}
For each closed subgroup $H$ of $G$, $\bar{P}_{H}$ with the differential
structure $C_{1}^{\infty }(\bar{P}_{H})$ is diffeomorphic to $P_{H}/G_{H}$.
\end{proposition}

\begin{proof}
The differential structure $C_2^{\infty }(\bar{P}_H)$ of $\bar{P}_H$
consists of pushforwards of $N_{H}$-invariant smooth functions on $P_{H}$.
However, a function $f\in C^{\infty }(P_{H})$ is $N_{H}$-invariant if and
only if it is $G_{H}$-invariant. But the differential structure of $P_{H}/G_{H}$ consists of $G_{H}$-invariant functions on $P_{H}$. Hence, the
differential structures of $\bar{P}_H$ and $P_{H}/G_{H}$ coincide.
\end{proof}

By Proposition \ref{1 copy(1)}, for each closed subgroup $H$ of $G$, the
flow $\exp tX$ of the invariant vector field $X$ preserves $P_{H}.$ The
politeness of the action of $G$ on $P$ ensures that $P_{H}$ is a manifold
and that the action of $G_{H}$ on $P_{H}$ is proper. Proposition \ref{0}
ensures that the action of $G_{H}$ on $P_{H}$ is free. Hence, $P_{H}/G_{H}$
is a quotient manifold of $P_{H}$, and $P_{H}$ has the structure of a left
principal $G_{H}$-bundle over $P_{H}/G_{H}.$ This implies that both the
reduction and the reconstruction of the restriction of $X$ to $P_{H}$ is the
same as in the case of a free and proper action. This completes the proof of
Theorem \ref{2}.

\vspace{20pt}

\noindent Larry M. Bates \newline
Department of Mathematics \newline
University of Calgary \newline
Calgary, Alberta \newline
Canada T2N 1N4 \newline
bates@ucalgary.ca

\vspace{20pt}

\noindent J\k{e}drzej \'Sniatycki \newline
Department of Mathematics \newline
University of Calgary \newline
Calgary, Alberta \newline
Canada T2N 1N4 \newline
sniatyck@ucalgary.ca

\end{document}